\documentclass{mcom-l}
\copyrightinfo{}{}
\makeatletter
\AtBeginDocument{\def\@serieslogo{\set@logo{\publname}}}
\makeatother

\usepackage{amsmath,amssymb,mathtools}
\usepackage{array}
\usepackage{booktabs}
\usepackage{graphicx}
\usepackage{multirow}
\usepackage{placeins}
\usepackage{url}
\usepackage[hidelinks]{hyperref}
\numberwithin{equation}{section}

\newtheorem{theorem}{Theorem}[section]
\newtheorem{lemma}[theorem]{Lemma}
\newtheorem{corollary}[theorem]{Corollary}
\newtheorem{proposition}[theorem]{Proposition}
\theoremstyle{definition}

\newtheorem{example}[theorem]{Example}
\theoremstyle{remark}
\newtheorem{remark}[theorem]{Remark}

\newcommand{\Ftwo}{\mathbb{F}_2}
\newcommand{\poly}{\Ftwo[x]}
\newcommand{\Dmin}{\Delta_{\min}}
\newcommand{\Wfb}{W_{\mathrm{fb}}}

\title[Reduction Modulo Binary Polynomials]{Reduction Modulo Binary
Polynomials with Logarithmic Feedback Depth}

\author{Junyu Zhou}
\address{Sichuan University}
\email{helloworld@junyu33.me}
\author{Kaiyi Zhang}
\address{Tsinghua University}
\email{kaiyizhang@tsinghua.edu.cn}
\subjclass[2020]{Primary 11T06; Secondary 12-08, 68W10, 68W40}
\keywords{binary polynomials, modular reduction, finite fields,
sparse polynomials, Frobenius map, parallel algorithms}

\begin{document}

\begin{abstract}
Polynomial modular reduction is central to binary finite-field arithmetic and
repeated Frobenius powering.  Top-down shift/XOR folding can use few operations
when the nonleading support is sparse, but a tap near the leading term creates
a long feedback chain.  We
formulate this recurrence as inversion of a nilpotent shift operator and factor
its inverse by characteristic-two Frobenius powers.  The resulting
\emph{Frobenius-factorized reduction} (FFR) applies to every monic binary
modulus without materializing a reciprocal or dense reduction matrix, and its
shifts can be generated online without a persistent modulus-specific schedule.
For degree $m$ and nonempty nonleading support of size $s$, with nearest-tap
distance $\Dmin$, FFR uses exactly $\lceil\log_2(m/\Dmin)\rceil$ sequential
feedback stages and has scheduled work $O(ms(1+\log(m/s)))$.  A portable-C
evaluation on 1,096 supports through degree
$131072$ identifies distinct FFR, L\'opez--Dahab, and gf2x-backed Barrett
regions.  On four certified irreducible moduli, FFR makes complete Rabin
irreducibility testing $1.35$--$8.04$ times faster than NTL and
$1.60$--$6.81$ times faster than the matched Barrett implementation.
\end{abstract}

\maketitle

\section{Introduction}
\label{sec:introduction}

Arithmetic in a polynomial-basis representation of a binary extension field
reduces products and squares modulo a monic polynomial.  The same operation
appears in polynomial factorization, irreducibility testing, and searches for
sparse irreducible or primitive polynomials.  In these settings the modulus is
public and often fixed for many reductions, but the useful degree, support,
and setup-amortization regimes vary substantially.

In the regime studied here, three modulus parameters provide the principal
coordinates for reduction time:
\begin{enumerate}
\item the degree $m$, which fixes the operand length and hence the amount of
      word-level data processed;
\item the Hamming weight $h$, including the leading term, which determines
      the number of nonleading taps processed by tapwise shift/XOR methods; and
\item the nearest-tap distance $\Dmin$, namely the gap from degree $m$ to the
      highest positive nonleading tap, when one exists, which controls how far
      feedback can propagate before leaving the degree-$m$ state; without
      such a tap there is no feedback chain.
\end{enumerate}
Observed runtime also depends on complete tap placement, word alignment,
setup, and the multiplication backend.  The three parameters above separate
the main size, sparsity, and feedback effects and therefore define the axes
and slices used in our evaluation.

When the nonleading part is sparse, top-down folding uses little setup and can
be extremely effective.  Its work and latency, however, are controlled by
different geometric features.  Hamming weight controls how many shifts and
XORs one fold introduces, whereas the distance from the leading term to the
nearest internal tap controls how far feedback can propagate.  Equal-degree,
equal-weight moduli may therefore induce very different dependency chains.

Barrett and Montgomery reduction avoid this particular serial recurrence, but
introduce multiplication kernels, modulus-dependent constants, and different
setup costs.  Reciprocal methods and dense linear maps expose additional
parallelism, but do not automatically preserve the original tap support.  We
ask whether the tap-induced feedback chain can be shortened without first
materializing a dense reciprocal or reduction matrix.

We answer this question with \emph{Frobenius-factorized reduction} (FFR).
It requires no materialized reciprocal, dense reduction matrix, or persistent
modulus-specific shift schedule: the doubled shifts may be derived online
from the original taps or retained in a lightweight reusable plan.  Online
execution uses an $m$-bit state and tap-sized descriptor scratch; ``no
persistent schedule'' refers specifically to precomputed reduction data.  For
an input of degree below twice the modulus degree, reduction gives a
finite forced recurrence.  We write it as $(I+U)Y=H$ for a truncated shift
operator $U$.  When $U\ne0$, characteristic two gives
\begin{equation}
  (I+U)^{-1}=\prod_{k=0}^{r-1}(I+U^{2^k}),
  \qquad
  r=\left\lceil\log_2\frac{m}{\Dmin}\right\rceil,
  \label{eq:intro-factorization}
\end{equation}
and $U^{2^r}=0$; when $U=0$, no feedback stage is needed.  Each Frobenius
power contains only the active original taps, with their distances doubled.

The finite geometric inverse underlying
\eqref{eq:intro-factorization} is classical
\cite{BiniPan1985,Kalorkoti1993}, but the identity alone does not give a
support-nondensifying executable reducer: expanding it exposes mixed feedback
paths and supplies neither a support-sensitive execution schedule nor
packed-word work bounds.
The algorithmic point here is that the Frobenius factors preserve the tap
structure without combinatorial densification inside a stage; mixed paths
arise only through composition across stages.  This turns the inverse into an
executable reducer for every monic binary modulus whose feedback depth, work,
schedule size, and packed-word execution can be analyzed directly from the
support geometry.

This construction generalizes our earlier Suwako reducer for trinomials
$x^m+x^t+1$~\cite{ZhouEtAl2025Suwako}.  That work established the
logarithmic-depth, precomputation-free trinomial case through two specialized
doubling folds.  With multiple taps, one cannot apply the trinomial fold to
each tap independently: all shifts in one factor must read the same stage
input, while their compositions create mixed paths between factors.  The
present contribution resolves this multi-tap execution and gives the resulting
all-tap geometry, packed-word analysis, planned and online implementations,
and comparative evaluation.

Our contributions are threefold.
\begin{enumerate}
\item We derive a nilpotent-feedback formulation whose characteristic-two
      factorization introduces no within-factor support densification and
      gives a correct FFR algorithm for any monic binary modulus.
\item We give exact tap-geometry analyses of coefficient work, packed-word
      work, feedback depth, schedule size, and workspace, together with
      planned and schedule-free online realizations.
\item We provide a reproducible portable-C evaluation that maps winning and
      losing regions against serial folding, L\'opez--Dahab, and gf2x-backed
      Barrett reduction, and test the reduction advantage inside a complete
      Rabin irreducibility test with an optimized NTL baseline.
\end{enumerate}

Correctness holds for every monic binary modulus, while sparsity governs the
performance regime.  Our depth results use the synchronous feedback-stage and
bounded-fan-in models defined below; complete-workload effects are evaluated
separately through Rabin irreducibility testing.

Section~\ref{sec:preliminaries} fixes the problem and cost models.
Sections~\ref{sec:operators}--\ref{sec:complexity} develop
the method and analysis.  Section~\ref{sec:related} states the prior-art
boundary.  Section~\ref{sec:evaluation} contains the implementation and
application evidence, and Section~\ref{sec:conclusion} concludes.

\section{Preliminaries and Cost Models}
\label{sec:preliminaries}

Let
\begin{equation}
  g(x)=x^m+q(x)=x^m+\bigoplus_{t\in T}x^t\in\poly,
  \qquad T\subseteq\{0,\ldots,m-1\},
  \label{eq:modulus}
\end{equation}
be monic, and set $s=|T|$ and $h=1+s$.  For $t\in T$, define
$\Delta_t=m-t$.  If $T\ne\varnothing$, let
\begin{equation}
  \Dmin=\min\{\Delta_t:t\in T\}.
\end{equation}
When $T\subseteq\{0\}$, the feedback operator below is zero and we set
$r=0$; otherwise $r$ is defined by~\eqref{eq:r}.

Identify degree-below-$m$ polynomials with vectors in $\Ftwo^m$.  Let $N$ be
the truncated right shift
\begin{equation}
  (N^d y)_i=\begin{cases}y_{i+d},&0\le i<m-d,\\0,&m-d\le i<m,
  \end{cases}
  \qquad N^m=0.
\end{equation}

For $A\in\poly$ with $\deg A<2m$, write
\begin{equation}
  A=L+x^mH,
  \qquad \deg L,\deg H<m.
  \label{eq:low-high}
\end{equation}
For degree-below-$m$ input $Y$, define linear maps $V$ and $U$ by
\begin{align}
  qY&=V(Y)+x^mU(Y),\nonumber\\
  V(Y)&=\sum_{t\in T}(x^tY\bmod x^m),
  &U&=\sum_{t\in T}N^{\Delta_t}.
  \label{eq:UV}
\end{align}
Here and below an empty sum denotes the zero map.  The term $t=0$, when
present, contributes the identity to $V$ and the zero shift $N^m$ to $U$.
Monicity is sufficient for a unique remainder; irreducibility is required only
when the quotient is intended to be a field.

We distinguish scheduled coefficient work, machine-word work, feedback depth,
bounded-fan-in gate depth, setup, and temporary space.  Scheduled coefficient
work counts shifted coefficient contributions, while machine-word work counts
the packed source-level data path. Machine latency and throughput are measured
under a fixed platform and timed boundary.
For the packed implementation, coefficient $i$ is bit $i\bmod W$ of word
$\lfloor i/W\rfloor$ in a little-endian array of
$n=\lceil m/W\rceil$ words.  An input representing $\deg A<2m$ has canonical
zero padding above coefficient $2m-1$; right-shift extraction preserves this
invariant, and output padding above coefficient $m-1$ is masked explicitly.
Caller-owned inputs and outputs are excluded from temporary-space counts.
Section~\ref{sec:evaluation} separately fixes
the steady-state reduction boundary, reusable-plan setup, storage, and any
amortization over a stated reuse count.

\section{Feedback Operators}
\label{sec:operators}

\begin{lemma}[Low/high decomposition]
\label{lem:decomposition}
For~\eqref{eq:low-high} and~\eqref{eq:UV},
\begin{equation}
  A-gY=L+V(Y)+x^m\bigl(H+Y+U(Y)\bigr).
\end{equation}
Consequently, if $H=(I+U)Y$, then $A\bmod g=L+V(Y)$.
\end{lemma}
\begin{proof}
Using~\eqref{eq:low-high} and~\eqref{eq:UV}, and using that subtraction and
addition coincide over $\Ftwo$, we obtain
\begin{align*}
  A-gY
  &=L+x^mH+x^mY+qY\\
  &=L+V(Y)+x^m\bigl(H+Y+U(Y)\bigr).
\end{align*}
If $H=(I+U)Y$, the entire high-part factor in the last display vanishes.
Hence $A-(L+V(Y))=gY$, while $\deg(L+V(Y))<m$.  The uniqueness of the
remainder on division by the monic polynomial $g$ proves the claim.
\end{proof}

\begin{lemma}[Support-nondensifying Frobenius powers]
\label{lem:frobenius}
For every $k\ge0$,
\begin{equation}
  U^{2^k}=\sum_{t\in T}N^{2^k\Delta_t},
  \label{eq:frobenius}
\end{equation}
where shifts by at least $m$ are zero.
\end{lemma}
\begin{proof}
The truncated shifts satisfy $N^aN^b=N^{a+b}$ and therefore commute.  In the
endomorphism algebra of $\Ftwo^m$, commuting elements satisfy
$(X_1+\cdots+X_s)^2=X_1^2+\cdots+X_s^2$.  Starting from
$U=\sum_{t\in T}N^{\Delta_t}$ and iterating this identity $k$ times gives
\[
  U^{2^k}=\sum_{t\in T}(N^{\Delta_t})^{2^k}
           =\sum_{t\in T}N^{2^k\Delta_t}.
\]
Finally, $N^d=0$ for $d\ge m$.  For $d<m$, the matrix of $N^d$ occupies the
$d$th superdiagonal.  Distinct surviving shifts therefore have disjoint
matrix support and cannot cancel.
\end{proof}

\begin{theorem}[Frobenius-factorized feedback inverse]
\label{thm:inverse}
If $U=0$, set $r=0$; then $(I+U)^{-1}=I$, with the empty product understood
as $I$.  Otherwise, for
\begin{equation}
  r=\left\lceil\log_2\frac{m}{\Dmin}\right\rceil,
  \label{eq:r}
\end{equation}
we have $U^{2^r}=0$ and
\begin{equation}
  (I+U)^{-1}=\prod_{k=0}^{r-1}(I+U^{2^k}).
  \label{eq:inverse}
\end{equation}
Moreover, $r$ is the least nonnegative integer for which $U^{2^r}=0$.
\end{theorem}
\begin{proof}
Suppose $U\ne0$.  Then $T$ contains a positive tap, $\Dmin<m$, and
$r\ge1$.  By the definition of $r$, every $t\in T$ satisfies
$2^r\Delta_t\ge2^r\Dmin\ge m$.  Lemma~\ref{lem:frobenius} therefore gives
$U^{2^r}=0$.  On the other hand,
$2^{r-1}\Dmin<m$, so the unique tap at distance $\Dmin$ contributes the
nonzero shift $N^{2^{r-1}\Dmin}$ to $U^{2^{r-1}}$.  Thus $r$ is minimal.

For any $j\ge0$, repeated use of $(I+X)^2=I+X^2$ gives the telescoping
identity
\begin{equation*}
  (I+U)\prod_{k=0}^{j-1}(I+U^{2^k})=I+U^{2^j}.
\end{equation*}
Taking $j=r$ makes the right-hand side $I$.  All factors are polynomials in
$U$ and hence commute, so the product is both a left and a right inverse of
$I+U$.  The case $U=0$ is immediate.
\end{proof}

This is a factored truncated reciprocal.  Formal-series inversion itself is
classical~\cite{Kalorkoti1993}; the proposed contribution is the
support-nondensifying fixed-length realization and its complete tap-geometry
analysis.

\section{Frobenius-Factorized Reduction}
\label{sec:algorithm}

FFR applies the Frobenius factors of the truncated feedback inverse without
materializing that inverse.  We distinguish a \emph{planned} realization,
which retains compact shift descriptors, from an \emph{online} realization,
which derives them from the original taps on every call.
Both implementation interfaces accept the complete nonleading support $T$ as
a duplicate-free list of exponents in strictly increasing order, and both
constructors validate this canonical representation.

For every $k\ge0$, define
\begin{equation}
  T_k=\{t\in T:2^k\Delta_t<m\}.
\end{equation}
Then $T_k=\varnothing$ for $k\ge r$.  At stage $k$, either realization uses
the shifts $2^k\Delta_t$ for $t\in T_k$.  A reusable plan stores these shifts
for $0\le k<r$ and the taps for final low assembly.  The online realization
instead enumerates $T_k$, doubles the active distances, and constructs the
same descriptors in tap order immediately before applying the stage; the
descriptors are overwritten at the next stage and are not reused between
reductions as a schedule.

\begin{quote}
\textbf{Algorithm 1 (Frobenius-factorized reduction).}
Split $A=L+x^mH$ and set $Y\leftarrow H$.  For $k=0,\ldots,r-1$, read one
immutable stage input $Y_{\rm old}$ and compute synchronously
\begin{equation*}
  Y_{\rm new}=Y_{\rm old}
  +\sum_{t\in T_k}N^{2^k\Delta_t}Y_{\rm old}.
\end{equation*}
Set $Y\leftarrow Y_{\rm new}$ and return $L+V(Y)$.
\end{quote}

An in-place tap loop that exposes updates within a stage computes a different
recurrence.

\begin{theorem}[Correctness]
\label{thm:correctness}
For every monic $g\in\poly$ of degree $m$ and every $A$ with $\deg A<2m$,
Algorithm~1 returns $A\bmod g$.
\end{theorem}
\begin{proof}
At stage $k$, every summand reads the same value $Y_{\rm old}$, so the stage
maps its input to
\[
  \left(I+\sum_{t\in T_k}N^{2^k\Delta_t}\right)Y_{\rm old}
  =(I+U^{2^k})Y_{\rm old}
\]
by Lemma~\ref{lem:frobenius}.  After all stages,
\[
  Y=\prod_{k=0}^{r-1}(I+U^{2^k})H=(I+U)^{-1}H
\]
by Theorem~\ref{thm:inverse}; the order of the factors is immaterial because
they commute.  Hence $H=(I+U)Y$, and Lemma~\ref{lem:decomposition} shows that
the returned polynomial $L+V(Y)$ is $A\bmod g$.
\end{proof}

\begin{example}
Let
\[
  g=x^{16}+x^{15}+x^{11}+x^3+1.
\]
The nonconstant taps have distances $1$, $5$, and $13$; the constant tap has
distance $16$ and contributes no feedback.  Thus $r=4$, and Algorithm~1
applies the four maps
\begin{align*}
  I+N+N^5+N^{13},\qquad
  I+N^2+N^{10},\qquad
  I+N^4,\qquad
  I+N^8.
\end{align*}
The shifts $26$, $20$, and $52$ have disappeared by truncation.  Their mixed
feedback paths have not been discarded: they arise implicitly when these
four factors are composed.  For the resulting state $Y$, the algorithm
returns
\[
  L+Y+(x^3Y\bmod x^{16})+(x^{11}Y\bmod x^{16})
    +(x^{15}Y\bmod x^{16}).
\]
\end{example}

\begin{remark}[Boundary cases]
If $T=\varnothing$, then $g=x^m$, $U=V=0$, no feedback stage is executed,
and the remainder is $L$.  If $T=\{0\}$, then $g=x^m+1$, $U=0$ and $V=I$,
so the remainder is $L+H$.  Omitting the constant tap merely removes the
identity summand from $V$; it does not change the argument.  Dense and
reducible moduli satisfy the same correctness theorem, although density can
make the schedule unattractive.  Word misalignment affects implementation
masking and storage, not the algebraic algorithm.
\end{remark}

Appendix~\ref{app:coefficient-algebra} gives the coefficient-algebra and
odd-characteristic extensions.

\section{Work--Depth--Space Geometry}
\label{sec:complexity}

We first count \emph{scheduled coefficient contributions}: adding one shifted
source coefficient to one destination coefficient costs one unit.  Retaining
the unshifted stage input, copying a vector, testing a public bound, and
reading a schedule descriptor are not included in this measure.  Let
$h_k=|T_k|$, and define
\begin{equation}
  \ell_t=\#\{k\ge0:2^k\Delta_t<m\}
  =\left\lceil\log_2\frac{m}{\Delta_t}\right\rceil .
  \label{eq:tap-lifetime}
\end{equation}
The equality includes the constant tap: if $t=0$, then $\Delta_t=m$ and
$\ell_t=0$.

\begin{theorem}[Geometry-sensitive scheduled work]
\label{thm:work}
The number of stored feedback shifts and the feedback contribution work are,
respectively,
\begin{align}
  S_{\rm fb}
    &=\sum_{k=0}^{r-1}h_k
      =\sum_{t\in T}\ell_t,
      \label{eq:schedule-size}\\
  \Wfb
    &=\sum_{t\in T}\sum_{k\ge0}[m-2^k\Delta_t]_+,
      \qquad [z]_+=\max\{z,0\}.
  \label{eq:work}
\end{align}
The final low-part assembly has scheduled coefficient work
\begin{equation}
  W_V=\sum_{t\in T}(m-t)=\sum_{t\in T}\Delta_t.
  \label{eq:low-work}
\end{equation}
If $s\ge1$, then
\begin{align}
  S_{\rm fb}
  &\le s+\log_2\frac{m^s}{s!}
   \le s\left(1+\log_2\frac{em}{s}\right),                 \label{eq:schedule-bound}\\
  \Wfb
  &\le mS_{\rm fb}
   =O\!\left(ms\left(1+\log\frac{m}{s}\right)\right),      \label{eq:work-bound}\\
  \Wfb+W_V
  &=O\!\left(ms\left(1+\log\frac{m}{s}\right)\right).
  \label{eq:total-shift-work}
\end{align}
The inequality in~\eqref{eq:work-bound} is strict whenever $U\ne0$.
\end{theorem}
\begin{proof}
Tap $t$ appears precisely for the $\ell_t$ integers $k$ satisfying
$2^k\Delta_t<m$, proving~\eqref{eq:schedule-size}.  At such an occurrence,
the right shift affects exactly $m-2^k\Delta_t$ coefficient positions.
Summing these contributions proves~\eqref{eq:work}.  Similarly,
$x^tY\bmod x^m$ contains $m-t=\Delta_t$ scheduled source positions, which
proves~\eqref{eq:low-work}.

Order the distinct distances as
$1\le d_1<\cdots<d_s\le m$.  Since $d_i\ge i$ and
$\lceil z\rceil\le z+1$,
\[
  S_{\rm fb}
  \le\sum_{i=1}^s\left(1+\log_2\frac{m}{i}\right)
  =s+\log_2\frac{m^s}{s!}.
\]
The standard inequality $s!\ge(s/e)^s$ gives the second part
of~\eqref{eq:schedule-bound}.  Every one of the $S_{\rm fb}$ feedback shifts
affects at most $m$ positions, and affects strictly fewer than $m$ positions
when it is nonzero.  This proves~\eqref{eq:work-bound}, including the
non-strict boundary $\Wfb=S_{\rm fb}=0$ for a constant-only tap set.  Finally,
$W_V\le ms$, and~\eqref{eq:total-shift-work} follows.
\end{proof}

\begin{corollary}[Feedback depth and a gate-depth upper bound]
\label{cor:depth}
If $U\ne0$, Algorithm~1 has exactly
\[
  r=\left\lceil\log_2\frac{m}{\Dmin}\right\rceil
\]
sequential feedback stages; if $U=0$, it has none.  Thus, for fixed $m$ and
$\Dmin$, feedback depth is independent of $s$.  In a fan-in-two XOR model
with free fanout and unrestricted temporary wires, the complete reduction has
the upper bound
\begin{equation}
  D_\oplus
  \le\sum_{k=0}^{r-1}\left\lceil\log_2(1+h_k)\right\rceil
    +\left\lceil\log_2(1+s)\right\rceil .
  \label{eq:gate-depth-upper}
\end{equation}
\end{corollary}
\begin{proof}
The exact feedback-stage count is the minimality statement in
Theorem~\ref{thm:inverse}.  Within stage $k$, each output coefficient is the
sum of its old value and at most $h_k$ shifted values, so a balanced XOR tree
has depth at most $\lceil\log_2(1+h_k)\rceil$.  The stages are composed
sequentially.  Final assembly adds $L$ and at most $s$ shifted values, giving
the last term in~\eqref{eq:gate-depth-upper}.
\end{proof}

Equation~\eqref{eq:gate-depth-upper} is a fan-in-two Boolean-circuit upper
bound. Feedback depth treats a whole synchronous vector update as one stage,
whereas gate depth accounts for the XOR tree inside that stage.

\subsection{Packed-word geometry}

Let $W$ be the machine-word width.  Put $n=\lceil m/W\rceil$.  The following
counts describe destination-word contributions; they are not instruction
counts.  A feedback shift by $d<m$ has potentially nonzero coefficients in
$\lceil(m-d)/W\rceil$ destination words.  A low-assembly shift by $t$ touches
\[
  n-\left\lfloor\frac{t}{W}\right\rfloor
\]
destination words.  The latter is not always
$\lceil(m-t)/W\rceil$: a non-word-aligned interval may intersect an additional
word.

\begin{proposition}[Word-level loop geometry]
\label{prop:word-geometry}
For $U\ne0$, set
\[
  a_k=\left\lceil\frac{m-2^k\Dmin}{W}\right\rceil
  \qquad(0\le k<r).
\]
The useful feedback-word support and the shifted-word contributions traversed
by the scalar stage loops are, respectively,
\begin{align}
  \widehat B_{\rm fb}
    &=\sum_{t\in T}\sum_{\substack{k\ge0\\2^k\Delta_t<m}}
       \left\lceil\frac{m-2^k\Delta_t}{W}\right\rceil,\\
  B_{\rm fb}^{\rm loop}
    &=\sum_{k=0}^{r-1}\sum_{t\in T_k}
      \min\left\{a_k,\,
        n-\left\lfloor\frac{2^k\Delta_t}{W}\right\rfloor\right\},\\
  B_V
    &=\sum_{t\in T}
       \left(n-\left\lfloor\frac{t}{W}\right\rfloor\right).
\end{align}
They satisfy
\begin{equation}
  \widehat B_{\rm fb}
  \le B_{\rm fb}^{\rm loop}
  \le \widehat B_{\rm fb}+S_{\rm fb}.
  \label{eq:feedback-padding-overhead}
\end{equation}
Moreover,
\begin{equation}
  \sum_{k=0}^{r-1}a_k\le B_{\rm fb}^{\rm loop}.
  \label{eq:stage-span-dominated}
\end{equation}
An in-place scalar implementation makes one $n$-word high-part extraction,
traverses $a_k$ destination words at feedback stage $k$, and makes one
$n$-word final assembly pass, in addition to the
$B_{\rm fb}^{\rm loop}+B_V$ shifted contributions.
\end{proposition}
\begin{proof}
A right shift by $d$ occupies coefficient positions $0,\ldots,m-d-1$, which
span $\lceil(m-d)/W\rceil$ words.  A left shift by $t$ occupies positions
$t,\ldots,m-1$, which span the words numbered
$\lfloor t/W\rfloor,\ldots,n-1$.  At stage $k$, the smallest active shift is
$2^k\Dmin$, so the union of all affected destinations spans exactly $a_k$
words.  For a particular feedback shift $d=2^k\Delta_t$, the scalar loop also
requires the source-word index $\lfloor d/W\rfloor$ to remain below the end of
the state.  It therefore visits
\[
  \min\{a_k,n-\lfloor d/W\rfloor\}
\]
destinations.  This is either the useful count
$\lceil(m-d)/W\rceil$ or that count plus one; the latter case performs only a
padding-word contribution.  Summing proves~\eqref{eq:feedback-padding-overhead}.
For each $k<r$, a tap attaining $\Dmin$ is active and contributes
\[
  \min\left\{a_k,
    n-\left\lfloor\frac{2^k\Dmin}{W}\right\rfloor\right\}=a_k.
\]
Summing these contributions proves~\eqref{eq:stage-span-dominated}; therefore
the complete scalar loop decomposition is
$O(n+B_{\rm fb}^{\rm loop}+B_V)$.
The extraction and assembly statements follow from traversing the $m$-bit
input and output states once.
\end{proof}

A word-aligned contribution can be formed from one source word.  A
non-word-aligned contribution can require two adjacent source words, two
machine shifts, and an XOR before accumulation.  Consequently,
Proposition~\ref{prop:word-geometry} records loop geometry; the latency of
aligned and cross-word contributions is ISA-specific and is determined by
measurement.

\begin{proposition}[Temporary space, online generation, and schedule setup]
\label{prop:space-setup}
Excluding the caller-owned input and output, the in-place scalar reduction
can be performed with one $m$-bit state.  The implementation evaluated in
this paper allocates exactly $n+1$ plan-owned scratch-array words: $n$ state
words and one zero sentinel.  An
explicit reusable plan stores
\[
  r\ \text{stage records},\qquad
  S_{\rm fb}\ \text{feedback-shift records},\qquad
  s\ \text{assembly-shift records},
\]
in addition to constant-size metadata.  Such a plan can be constructed in
$O(s+r+S_{\rm fb})$ word-RAM operations and the same asymptotic number of
records.  Alternatively, the factors can be generated online with no
persistent feedback or assembly schedule.  This realization uses $n+1$ state words and
$s$ reusable descriptor records, hence $O(n+s)$ workspace, and performs
$O(s+r+S_{\rm fb})$ additional scalar descriptor-generation operations per
reduction.
\end{proposition}
\begin{proof}
For a positive right shift, output coefficient $i$ depends only on old
coefficients with indices at least $i$.  Updating coefficients from low to
high therefore never destroys a value needed by a later destination.  This
permits one in-place state; the extra sentinel is an implementation device for
branch-free adjacent-word reads.

The record counts follow directly from~\eqref{eq:schedule-size}.  For the
setup bound, compute every $\ell_t$ by repeated doubling, at total cost
$O(s+S_{\rm fb})$.  Because each tap is active on the stage interval
$0,\ldots,\ell_t-1$, a difference array followed by one prefix sum computes
all $r$ per-stage counts in $O(s+r)$ operations.  A second pass emits the
$\ell_t$ doubled shifts of each tap into the corresponding stage buckets in
$O(S_{\rm fb})$ operations.

For the online realization, keep one $s$-record descriptor array.  Because
the canonical taps are ordered, traversing them from highest to lowest emits
the active feedback distances in increasing order and stops at the first
inactive distance.  The total number emitted is $S_{\rm fb}$, with at most one
termination check per stage; final assembly emits $s$ descriptors.  This gives
$O(s+r+S_{\rm fb})$ scalar control work without carrying a descriptor set
forward as the schedule for the next reduction.
\end{proof}

The setup bound describes an available two-pass schedule construction, not
compilation of generated code.  The benchmarked reusable-plan builder uses a
simpler stage-by-stage scan, sort, and allocation strategy and is not claimed
to realize this asymptotic setup bound.  The experiments time that concrete
builder and the schedule-free workspace separately, and measure online
descriptor generation inside reduction.

For a compact comparison, let $M(n)$ and $D_M(n)$ denote the work and
dependency depth of an $n$-word polynomial multiplication.  For $U\ne0$, let
$\nu_U=\min\{j\ge1:U^j=0\}$; then
$\nu_U\le\lceil m/\Dmin\rceil$.  The boundary $U=0$ has no dependent
feedback fold.  The methods use different cost primitives, so their constants
and crossovers are determined experimentally.
Table~\ref{tab:method-complexity} collects the resulting method-specific
source-level costs; Section~\ref{sec:evaluation} supplies the corresponding
measured times.

\begin{table}[t]
\caption{Reduction work and dependency depth.  Here
$n=\lceil m/W\rceil$, and fan-in-two depth assumes free fanout.
As defined in Proposition~\ref{prop:word-geometry},
$B_{\rm fb}^{\rm loop}$ counts shifted-word contributions traversed by the
feedback loops, while $B_V$ counts those traversed by final low-part assembly.
The serial row assumes $U\ne0$, and the L\'opez--Dahab row assumes
$m>W$ and $\deg q\le m-W$.  The displayed depths retain the natural dependency
measure of each method rather than defining a common gate-depth metric.}
\label{tab:method-complexity}
\centering
\footnotesize
\renewcommand{\arraystretch}{1.35}
\begin{tabular}{@{}
  >{\raggedright\arraybackslash}p{0.17\textwidth}
  >{\centering\arraybackslash}p{0.34\textwidth}
  >{\centering\arraybackslash}p{0.41\textwidth}@{}}
\toprule
Method & Work & Dependency depth\\
\midrule
FFR
& $C_{\rm word}=O\!\left(n+B_{\rm fb}^{\rm loop}+B_V\right)$
& $D_{\rm fb}=r=\left\lceil\log_2(m/\Dmin)\right\rceil$\\
\midrule
Serial
& $C_{\rm word}=O(sn\nu_U)$
& $D_{\rm fb}\le\nu_U\le\left\lceil m/\Dmin\right\rceil$\\
\midrule
L\'opez--Dahab
& $C_{\rm word}=\Theta\!\left(n(1+s)\right)$
& $D_{\rm loop}=n+O(1)$\\
\midrule
BarrettGF2X
& $C_{\rm word}=2M(n)+O(n)$
& $D=2D_M(n)+O(1)$\\
\midrule
Dense map
& $C_{\rm word}=\Theta(mn)=\Theta(m^2/W)$
& $D=O(\log m)$\\
\bottomrule
\end{tabular}
\renewcommand{\arraystretch}{1}
\end{table}

\subsection{Uniform fixed-weight supports}

We next record what the deterministic geometry implies for one explicit
distribution.  Fix $m$ and $1\le s\le m$, and choose $T$ uniformly among the
$s$-subsets of $\{0,\ldots,m-1\}$.  Equivalently, the distances form a
uniform $s$-subset of $\{1,\ldots,m\}$.  This model permits a zero constant
coefficient and reducible moduli; it is not a distribution on irreducible
polynomials.

\begin{theorem}[Random-support schedule geometry]
\label{thm:random-support}
Under the uniform fixed-weight model,
\begin{align}
  \mathbb E[h_k]
    &=\frac{s}{m}\left\lfloor\frac{m-1}{2^k}\right\rfloor
      <\frac{s}{2^k},                                      \label{eq:expected-hk}\\
  \mathbb E[S_{\rm fb}]&<2s,                               \label{eq:expected-schedule}\\
  \mathbb E[r]&\le\lceil\log_2s\rceil+2,                   \label{eq:expected-depth}\\
  \mathbb E[\Wfb]&<2ms.                                    \label{eq:expected-work}
\end{align}
Moreover, $\mathbb E[W_V]=s(m+1)/2$, and hence the expected total scheduled
coefficient contribution work is $O(ms)$.
\end{theorem}
\begin{proof}
There are exactly $\lfloor(m-1)/2^k\rfloor$ distances $d\in\{1,\ldots,m\}$
with $2^kd<m$.  Each distance belongs to the random support with probability
$s/m$, proving~\eqref{eq:expected-hk}.  Summing over $k\ge0$ and using
$S_{\rm fb}=\sum_kh_k$ gives~\eqref{eq:expected-schedule}.

Since $r=\sum_{k\ge0}\mathbf 1_{\{h_k>0\}}$, Markov's inequality and
\eqref{eq:expected-hk} give
\[
  \Pr(h_k>0)\le\min\{1,s/2^k\}.
\]
Splitting the sum at $k=\lceil\log_2s\rceil$ and summing the geometric tail
proves~\eqref{eq:expected-depth}.  The deterministic inequality
$\Wfb\le mS_{\rm fb}$ proves~\eqref{eq:expected-work}.  Finally, every distance
has inclusion probability $s/m$, so
\[
  \mathbb E[W_V]
  =\frac{s}{m}\sum_{d=1}^m d
  =\frac{s(m+1)}{2}.
\]
\end{proof}

Exact work depends on every tap distance beyond the summary coordinates
$(h,\Dmin)$.  The bounds above therefore describe families and schedules,
while method selection at fixed $(m,h,\Dmin)$ can still vary with the complete
tap set.

\section{Prior Art and Novelty Boundary}
\label{sec:related}

Reciprocal and triangular Toeplitz inversion for polynomial division are
classical~\cite{BiniPan1985}, as are formal-power-series inversion
\cite{Kalorkoti1993}.  Parallel triangular Toeplitz and recurrence solvers
already use doubling, bandwidth, or prefix structure
\cite{Bini1984,Morf1980,KoggeStone1973,HoLee1990}. These works establish the
classical status of finite geometric inversion and logarithmic-depth
recurrence solving. Our question is whether a support-nondensifying fixed-state
realization retains useful support geometry without dense reciprocals,
polynomial multiplication, or expanded state.

The direct predecessor of the present method is our Suwako algorithm for
trinomials $x^m+x^t+1$~\cite[Lemma~2, Eq.~(5)]{ZhouEtAl2025Suwako}.  It already replaces the
$\Theta(m/(m-t))$ serial feedback chain by logarithmically many doubling steps
without modulus-specific precomputation.  The present paper does not claim
that trinomial mechanism anew: it derives the multi-tap operator equation,
shows that every Frobenius factor contains only the surviving doubled images
of the original taps, and develops exact geometry and packed implementations
for arbitrary monic binary moduli.

The closest external recurrence-level antecedent is the term-preserving look-ahead transformation
for LFSRs: in characteristic two, replacing a generator by its square doubles
the feedback distances without increasing its term count~\cite{LinEtAl2013}.
Parallel LFSR and CRC architectures also construct equivalent state-space
updates for arbitrary generator polynomials~\cite{AyinalaParhi2011}. These
provide the nearest recurrence-level antecedents for term-preserving doubling
and parallel feedback. They transform or parallelize recurrence/state-space
updates, whereas Algorithm~1 applies successive factors of an implicit inverse
to a forced length-$m$ triangular system on a fixed $m$-bit state.

Sparse binary-field reduction includes word- and bit-parallel polynomial-basis
methods~\cite{Wu1999} and strong specialized formulae
\cite{LopezDahab2000}.  Niehues, Custodio, and Panario give a uniform
top-down reducer for arbitrary low-weight moduli and explicitly observe that
tap placement can induce a long critical path~\cite{NiehuesEtAl2018}.
Finite-field multiplier architectures additionally use multi-degree
reduction, subexpression sharing, and balanced XOR trees~\cite{Meher2009}.
Our measured comparison accordingly keeps serial folding and applicable
L\'opez--Dahab reduction as direct baselines.

Barrett and Montgomery provide multiplication-based alternatives, including
restricted precomputation-free families~\cite{KnezevicEtAl2008}. FFR differs
in supporting arbitrary monic binary moduli without a materialized reciprocal
or dense matrix: its planned realization has a separately measured compact
schedule, while its online realization retains only reusable state and
descriptor scratch and generates shifts inside each reduction.

Repeated Frobenius powering in factorization gives the closest application
context.  Von zur Gathen and Gerhard combine theory, implementation, storage,
and complete factorization over $\Ftwo$~\cite{GathenGerhard2002}; irreducible
polynomial construction gives another context~\cite{Shoup1990}.

Relative to these antecedents, FFR combines arbitrary monic-modulus
correctness, support-nondensifying Frobenius factors, logarithmic feedback depth,
execution without a materialized dense reciprocal or matrix, and exact
geometry from all tap distances.

\section{Implementations and Evaluation}
\label{sec:evaluation}

\subsection{Implementations and validation}

The comparison uses the direct operation $A\mapsto A\bmod g$ for
$\deg A<2m$.  Serial folding repeatedly eliminates the high part through the
nonleading support; the ordinary-loop L\'opez--Dahab method instead uses
word-oriented shift/scatter operations where applicable.  BarrettGF2X uses
polynomial products and a modulus-dependent reciprocal, whereas the dense
baseline materializes a fixed-modulus linear map.  Table~\ref{tab:families}
summarizes their applicability and role.

\begin{table}[t]
\caption{Compared direct-reduction methods.  Throughout, $g$ is as
in~\eqref{eq:modulus} and $\deg A<2m$.}
\label{tab:families}
\vspace{0.5em}
\centering
\footnotesize
\renewcommand{\arraystretch}{1.12}
\begin{tabular}{@{}
  >{\centering\arraybackslash}m{0.16\textwidth}
  >{\centering\arraybackslash}m{0.12\textwidth}
  >{\centering\arraybackslash}m{0.20\textwidth}
  >{\raggedright\arraybackslash}m{0.435\textwidth}@{}}
\toprule
Method & Status & Applicability & Use in evidence\\
\midrule
\multirow{2}{*}{FFR} & Planned & \eqref{eq:modulus}
& Direct candidate at all 1,096 points.\\
\cmidrule(l){2-4}
& Online & \eqref{eq:modulus}
& Schedule-free 45-point study.\\
\midrule
Serial folding & Native & \eqref{eq:modulus}
& Direct on the 392-point subset.\\
\midrule
BarrettGF2X & Native/gf2x & \eqref{eq:modulus}
& Direct at all 1,096 points.\\
\midrule
L\'opez--Dahab & Native & $m>W$, $\deg q\le m-W$
& Direct at every measured applicable point.\\
\midrule
Dense map & Native & $g$ fixed
& Diagnostic map; $\Theta(m^2/W)$ words and row-word work; 64 MiB cap.\\
\bottomrule
\end{tabular}
\renewcommand{\arraystretch}{1}
\end{table}

Generated fixed-modulus code and Montgomery REDC are excluded: the former is
a specialization contract, while the latter computes $AR^{-1}\bmod g$ rather
than direct reduction.
The planned and online FFR realizations share scalar kernels, while FFR and
serial folding share the packed-word shift/XOR primitive.  Before timing,
every timed binary reducer is compared with independent long division;
Appendix~\ref{app:validation} records these checks and the extended algebraic
suites.

\subsection{Experimental contract}
\label{sec:experimental-contract}

We measured portable scalar C on one recorded x86-64 host, pinned to logical
CPU~0 and compiled with GCC~16.1.1 using
\texttt{-O3 -std=c11 -Wall -Wextra}; Barrett linked
\texttt{/usr/lib/libgf2x.so.3.0.0}.  Every row records the governor,
energy-performance preference, and turbo state, and conclusions are scoped to
this platform.

Each support uses eight materialized inputs $A=L+x^mH$ with independent uniform
$m$-bit halves.  After one warm-up, 31 or 127 complete trials use twelve batch
repetitions and cyclic method order.  Only \texttt{reduce\_into} is timed;
setup, input generation, allocation, checking, and output handling are
excluded, and no observation is removed.

The controlled corpus contains 1,096 constant-free supports at
\[
  m\in\{128,512,2048,8192,32768,131072\}.
\]
Put
\[
  D_m=\{2^j:0\le j\le\log_2(m/2)\},\qquad
  H_0=\{2,3,5,9,17,33,65\},
\]
and set $H_m=H_0\cup E_m$, where the frozen high-weight additions are
\begin{align*}
E_{128}&=\{81,97,113,121\},\\
E_{512}&=\{97,129,193,257,385\},\\
E_{2048}=E_{8192}&=\{97,129,193,257,385,513,769,1025\},\\
E_{32768}&=\{97,129,161,193,225,257,321,385,513,769,1025\},\\
E_{131072}&=\{97,129,193,257,385,513,641,769,897,1025\}.
\end{align*}
For every $(h,\Dmin)\in H_m\times D_m$ satisfying
$h-1\le m-\Dmin$, the corpus contains the single support
\[
  T(m,h,\Dmin)
  =\left\{\left\lfloor\frac{j(m-\Dmin)}{h-1}\right\rfloor:
    1\le j\le h-1\right\}.
\]
This construction fixes the nearest tap at $m-\Dmin$ and yields exactly 1,096
distinct supports.  The sets $E_m$ are grid-level crossover refinements fixed
after preliminary screening; no individual support was selected by its
measured winner.

Every point compares FFR with Barrett, and L\'opez--Dahab is present exactly at
measured points with $\Dmin\ge64$.  Serial is retained on 392 manifest-defined
points: at all six degrees, the complete grid
$H_0\times\{1,2,4,8,16,32\}$, together with, for
$m\in\{512,2048,8192,32768\}$, the grid
$H_0\times\{2^j:128\le2^j\le m/2\}$.  These are complete grids rather than
pointwise runtime selections.

All remaining cohorts use a no-Serial contract
because its dependent iteration count can reach $\lceil m/\Dmin\rceil$ and its
tapwise work grows with $h$, making full extension impractical.
Winners range only over methods measured at that point.

Each method is summarized by its trial median and a deterministic
10,000-resample bootstrap 95\% interval.  A unique winner requires relative
median-interval half-width at most 1\%, a margin of at least 1\% over every
competitor, and paired ratio intervals strictly below one; failures remain
timing instability, operational ties, or statistical ties.

\FloatBarrier
\subsection{Setup, storage, and online generation}

\begin{table}[t]
\caption{Median setup time and requested plan-owned bytes at
$(h,\Dmin)=(9,1)$.  Each entry is ``nanoseconds / bytes.''}
\label{tab:setup-storage}
\centering
\small
\begin{tabular}{@{}rrrrr@{}}
\toprule
$m$ & trials & FFR & Serial & Barrett\\
\midrule
128    & 127 & 596.5 / 744    & 119.5 / 416   & 782.0 / 184\\
\midrule
512    & 127 & 719.5 / 888    & 123.5 / 512   & 4770.5 / 376\\
\midrule
2048   & 127 & 887.5 / 1176   & 137.0 / 896   & 50799.0 / 1144\\
\midrule
8192   & 127 & 1152.0 / 2040  & 149.0 / 2432  & 558369.5 / 4216\\
\midrule
32768  & 127 & 1312.0 / 5208  & 235.0 / 8576  & 7760258.5 / 16504\\
\midrule
131072 & 31  & 2832.5 / 17592 & 856.5 / 33152 & 119030566.0 / 65656\\
\bottomrule
\end{tabular}
\end{table}

Table~\ref{tab:setup-storage} reports setup and requested plan-owned storage
under separately frozen contracts.
The storage counts exclude allocator overhead, shared modulus storage,
benchmark buffers, and transient library workspace.
The setup measurements show a tradeoff hidden by steady-state timings:
Barrett's reciprocal construction becomes expensive at large $m$, while FFR
retains a modest schedule-construction cost.  For an explicitly chosen reuse
count $K$, the artifact separately derives
\[
  T_{\mathrm{total}}(K)=T_{\mathrm{setup}}+K T_{\mathrm{reduce}};
\]
amortized values are reported only for explicitly stated reuse counts $K$.

The online/planned study uses complete constant-free Cartesian slices:
$h\in\{3,9,65\}$ for $m=128$ and $h\in\{3,9,65,513\}$ for
$m\in\{2048,32768,131072\}$, each crossed with
$\Dmin\in\{1,64,m/4\}$.  The resulting $9+36=45$ supports use the
deterministic spread rule above.

Planned and online FFR received the same eight inputs, used the same scalar
kernels, and
were cyclically ordered over 31 retained trials; the online path regenerated
all descriptors inside the timed reduction.  The median
$T_{\rm online}/T_{\rm planned}$ over each degree slice was $1.676$, $1.086$,
$1.006$, and $1.002$, respectively.  Thus descriptor generation is material
at small degree but is within about one percent of planned steady-state time
in the two largest slices; five individual paired intervals contain one.  If
setup and one reduction are combined as the declared derived quantity
$T_{\rm setup}+T_{\rm reduce}$, online FFR is lower at all 45 points and has
overall median ratio $0.915$.  This $K=1$ value is derived from the separately
measured components.  On this portable-C platform, online descriptor
generation therefore approaches planned FFR performance at large degrees,
while online setup still includes workspace allocation.

\subsection{Measured winner regions}

Figure~\ref{fig:winner-panels} gives the complete measured phase diagram.  Its
axes are the two principal geometric parameters rather than ordinal grid
indices.  Every block is one observed support; blank regions were not measured
and are not interpolated.  To expose the predicted regime decomposition, we
overlay a fitted cost model.  At each fixed $m$, a deterministic alternating
split of the sorted grid fits affine times for FFR from its portable scalar
block-work count and for L\'opez--Dahab from its shift/scatter count; Barrett
is represented by its fixed-$m$ calibration median.  More precisely, the
predictors and fitted times are
\begin{align}
  C_{\rm FFR}&=\widehat B_{\rm fb}+B_V,
  &\widehat\tau_{\rm FFR}&=\alpha_{{\rm FFR},m}
    +\beta_{{\rm FFR},m}C_{\rm FFR},\notag\\
  C_{\rm LD}&=n(h-1)=ns,
  &\widehat\tau_{\rm LD}&=\alpha_{{\rm LD},m}
    +\beta_{{\rm LD},m}C_{\rm LD},\notag\\
  &&\widehat\tau_{\rm B}&=b_m,
  \label{eq:winner-time-model}
\end{align}
where each $\alpha$ and $\beta$ is fitted on the calibration half and $b_m$
is the Barrett calibration median.  Here $C_{\rm FFR}$ is a useful-block-work
predictor, not the exact executed loop count $B_{\rm fb}^{\rm loop}+B_V$;
their difference is the padding contribution bounded in
Proposition~\ref{prop:word-geometry}.  If $\mathcal A(p)$ is the set of methods
available at grid point $p=(h,\Dmin)$, its predicted region for method $j$ is
defined by
\begin{equation}
  \mathcal R_j
  =\{p:\widehat\tau_j(p)\le\widehat\tau_i(p)
       \text{ for every }i\in\mathcal A(p)\}.
  \label{eq:winner-region-inequalities}
\end{equation}
The minimum predicted time labels each measured coordinate, and a line is
drawn only where adjacent coordinates belong to different predicted regions.
Thus the blocks remain measurements whereas the lines are discrete,
descriptive approximations to the pairwise interfaces induced
by~\eqref{eq:winner-region-inequalities}.

\begin{figure}[t]
\centering
\includegraphics[width=\textwidth]{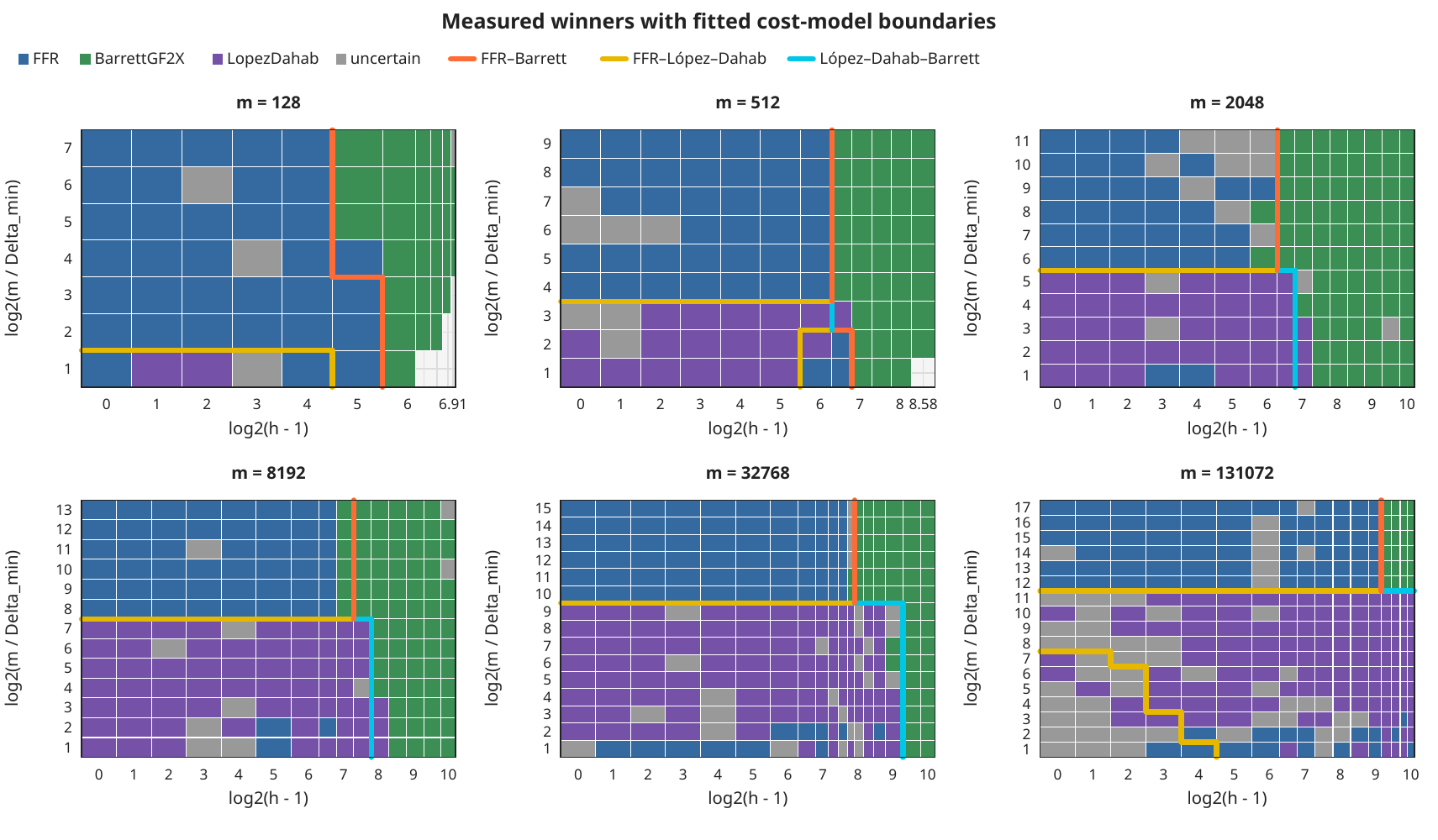}
\caption{Measured winner regions for six fixed degrees.  The horizontal axis
is $\log_2(h-1)$ and the vertical axis is
$\log_2(m/\Dmin)$.  Colors identify the unique winner among the methods
available at that point; gray cells are uncertain.  FFR and Barrett are present
everywhere, whereas L\'opez--Dahab is admitted only at its tested
$\Dmin\ge64$ boundary.  Thick lines mark pair-specific interfaces predicted
by the fixed-$m$ model~\eqref{eq:winner-time-model}, separating regions defined
by the inequalities~\eqref{eq:winner-region-inequalities}. The lines connect
only adjacent measured coordinates and leave the measured classifications
unchanged.}
\label{fig:winner-panels}
\end{figure}

Of the 1,096 cells, 983 have a unique winner: FFR wins 321, Barrett wins 305,
and L\'opez--Dahab wins 357.  The remaining 113 cells comprise 103
timing-unstable cells and 10 operational ties; no cell is classified as a
statistical tie.  The largest panel, $m=131072$, deliberately retains only 31
trials and consequently has 52 unstable cells out of 289; these cells remain
unassigned in the phase diagram.

The diagram separates three regimes.  FFR occupies the sparse,
difficult-feedback region.  As weight increases, multiplication-based Barrett
eventually wins even when $\Dmin$ is small.  When $\Dmin\ge64$,
L\'opez--Dahab increasingly dominates the friendly-feedback region as $m$
grows.  Across rows with $\Dmin<64$, the observed persistent Barrett onset is
$h=33$--$65$, $97$, $65$--$97$, $129$, $225$--$257$, and $641$ for the six
successive degrees.  In the L\'opez--Dahab-applicable rows the corresponding
onsets are $65$, $129$, $129$--$193$, $257$--$385$, $513$--$769$, and beyond
the largest sampled weight $1025$.  These are the sampled crossover locations
on the recorded platform.

The $\Dmin=1$ slice permits a direct check of the fitted FFR--Barrett
interface.  For its controlled support, write
$C_m(h):=C_{\rm FFR}(T(m,h,1))
=\widehat B_{\rm fb}(T(m,h,1))+B_V(T(m,h,1))$, using the exact sums in
Proposition~\ref{prop:word-geometry}.
Since the fitted slopes are positive, the model assigns the point to Barrett
rather than FFR exactly when
\begin{equation}
\begin{aligned}
  b_m\le\alpha_{{\rm FFR},m}+\beta_{{\rm FFR},m}C_m(h)
  &\quad\Longleftrightarrow\quad
  C_m(h)\ge C_m^\star
  :=\left\lceil\frac{b_m-\alpha_{{\rm FFR},m}}
                         {\beta_{{\rm FFR},m}}\right\rceil,\\
  \widehat h_m&:=\min\{h\in H_m:C_m(h)\ge C_m^\star\}.
\end{aligned}
  \label{eq:barrett-ffr-crossover}
\end{equation}
Table~\ref{tab:portable-results} compares this prediction with the measured
persistent crossover.  Five of the six degrees agree exactly; at $m=8192$
the model predicts the next sampled weight, $193$ instead of $129$.

\begin{table}[t]
\caption{Crossover summary on the constant-free $\Dmin=1$ slice.  FFR time
and Barrett/FFR use $h=9$; model thresholds follow
\eqref{eq:barrett-ffr-crossover}, and $h_m^{\rm obs}$ is the measured
persistent threshold.}
\label{tab:portable-results}
\centering
\small
\begin{tabular}{@{}rrrrrr@{}}
\toprule
$m$ & FFR (ns) & Barrett/FFR & $C_m^\star$ & $\widehat h_m$ & $h_m^{\rm obs}$\\
\midrule
128    & 64.1    & 2.164  & 123     & 33  & 33\\
\midrule
512    & 272.9   & 3.807  & 945     & 97  & 97\\
\midrule
2048   & 1079.2  & 4.247  & 3631    & 97  & 97\\
\midrule
8192   & 4390.2  & 8.654  & 26356   & 193 & 129\\
\midrule
32768  & 17008.4 & 16.768 & 192509  & 257 & 257\\
\midrule
131072 & 68167.1 & 37.712 & 1804748 & 641 & 641\\
\bottomrule
\end{tabular}
\end{table}

\subsection{Does the work model predict FFR time?}

Figure~\ref{fig:work-runtime} compares the formal scheduled coefficient work
$\Wfb$ with measured FFR time.  Within each fixed degree, the Spearman
correlation ranges from $0.9868$ to $0.9934$.  Thus the complete tap geometry
captured by $\Wfb$ orders these controlled scalar runtimes unusually well.
The residual scatter records effects beyond the source-work model, while the
correlation summarizes runtime ordering under that model.

\begin{figure}[b]
\centering
\includegraphics[width=\textwidth]{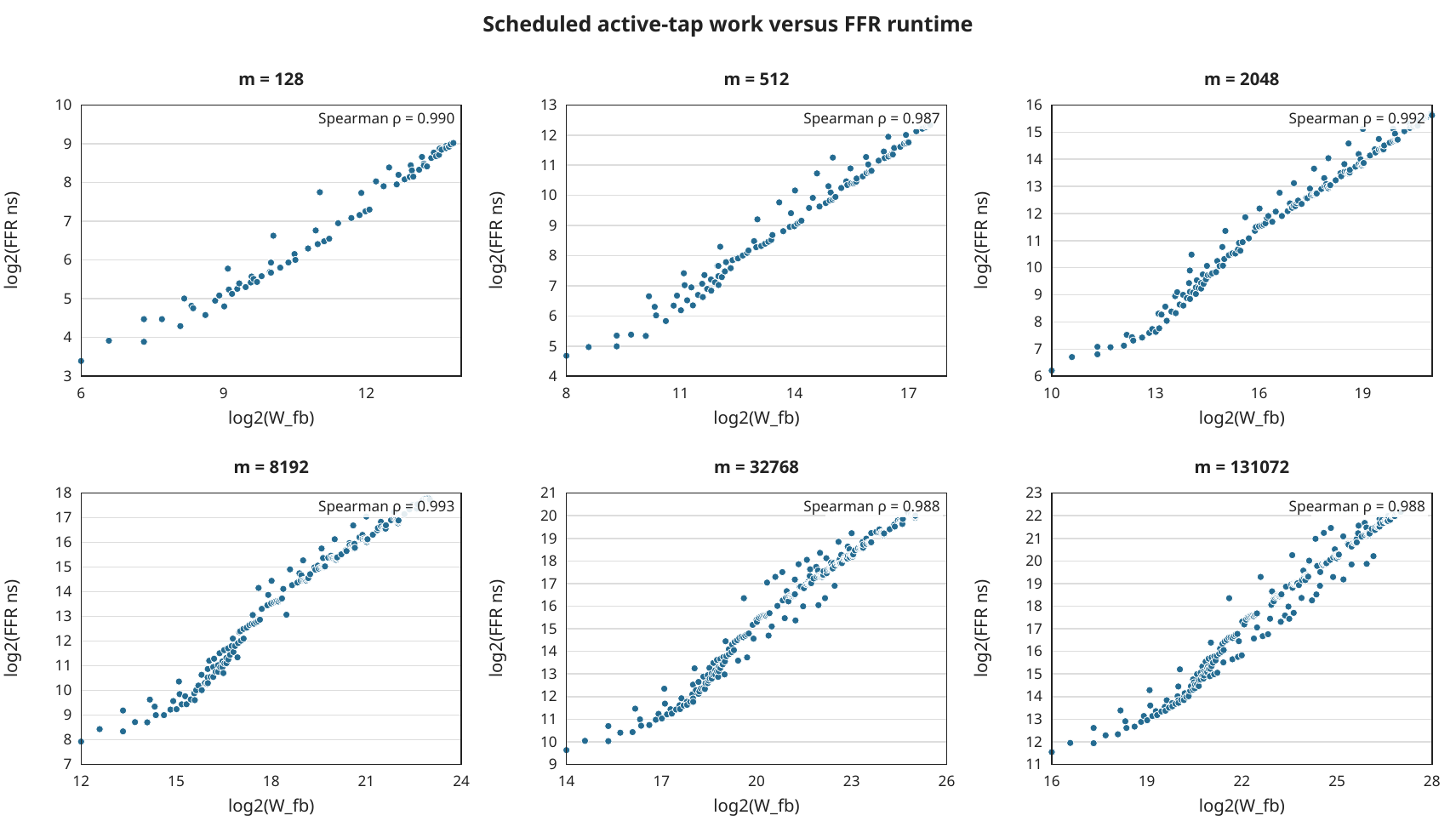}
\caption{Scheduled active-tap work versus median FFR reduction time for all
1,096 controlled supports, shown separately at fixed $m$.  Both axes use
base-two logarithms.  The displayed $\rho$ is the within-panel Spearman
correlation.}
\label{fig:work-runtime}
\end{figure}

\FloatBarrier
\subsection{End-to-end irreducibility testing}

We finally test whether the reduction result survives inside one complete
computational-algebra workload.  For a power-of-two degree $m$, Rabin's
criterion computes the chain of $m$ modular squares beginning at $x$, checks
\[
  \gcd\bigl(x^{2^{m/2}}-x,g\bigr)=1,
  \qquad
  x^{2^m}=x\pmod g.
\]
We deterministically selected one Sage-certified irreducible modulus at each
$m\in\{128,512,2048,8192\}$, always with $(h,\Dmin)=(9,1)$.  At each degree,
the generator initializes a pseudorandom generator with the fixed seed recorded
in the manifest, includes taps $0$ and $m-1$, samples the other six taps
uniformly without replacement from $\{1,\ldots,m-2\}$, skips repeated
supports, and retains the first candidate for which Sage
\texttt{is\_irreducible()} returns true; the accepted-attempt index is also
recorded.  No
timing result enters this search.  Each matched path uses the same scalar
bit-dilation squarer and NTL GCD code and changes only the reducer.  Its primary
interval includes reducer setup, all modular
squares, the GCD, and the final equality test.  NTL
\texttt{IterIrredTest} is timed separately as an optimized complete-library
baseline.

\begin{figure}[b]
\centering
\includegraphics[width=\textwidth]{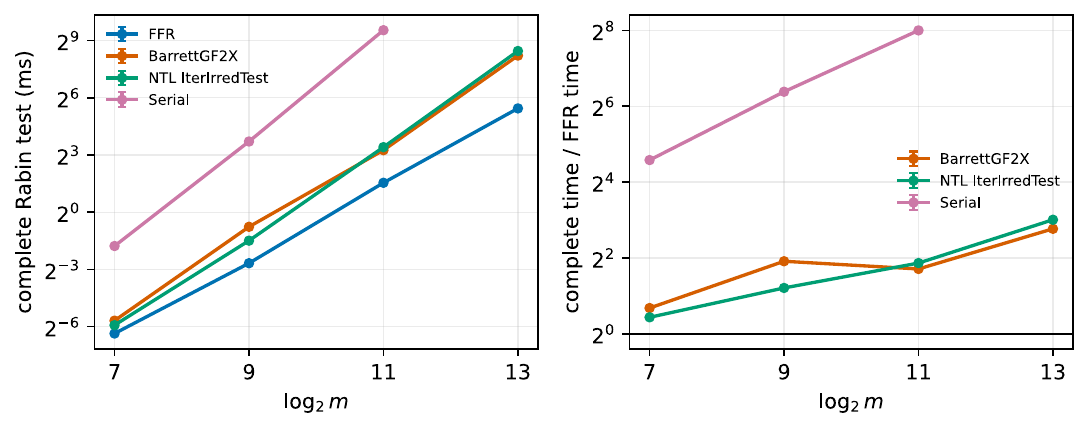}
\caption{Complete Rabin irreducibility testing on four deterministically
selected, Sage-certified sparse moduli with $(h,\Dmin)=(9,1)$.  The left panel
shows median end-to-end time with bootstrap 95\% intervals; the right panel
normalizes each complete implementation by FFR.  Every point contains 31
trials.  Serial has no $m=8192$ point because a retained smoke trial took about
47.5 seconds.}
\label{fig:rabin-e2e}
\end{figure}

Figure~\ref{fig:rabin-e2e} shows a stable end-to-end effect.  Across increasing
$m$, NTL takes $1.35$, $2.31$, $3.64$, and $8.04$ times the FFR time; the
matched Barrett path takes $1.60$, $3.76$, $3.27$, and $6.81$ times the FFR
time.  All paired ratio intervals lie strictly above one.  The modular-square
chain accounts for $85.5\%$ of FFR time at $m=128$ and $98.9\%$ at $m=8192$,
so the measured difference remains central to the complete test rather than
being hidden by setup or the terminal GCD.

\FloatBarrier
\section{Conclusion}
\label{sec:conclusion}

We recast reduction by an arbitrary monic binary polynomial as inversion of a
nilpotent feedback operator.  Characteristic-two Frobenius powers yield exact
feedback depth $\lceil\log_2(m/\Dmin)\rceil$ and work governed by every tap
distance without materializing a reciprocal or dense reduction matrix.  The
doubled shifts may be retained in a reusable plan or generated online with
$O(n+s)$ workspace and no persistent modulus-specific schedule.  The analysis
separates coefficient work, packed-word loop geometry, descriptor-generation
work, setup, and temporary space.

The portable scalar experiments show that this tradeoff has nontrivial
algorithm-selection regions.  FFR wins for sparse moduli with
difficult feedback, gf2x-backed Barrett takes over as weight grows, and
ordinary-loop L\'opez--Dahab is often strongest when its feedback-distance
condition applies.  The formal work measure is strongly monotone with FFR
runtime on the controlled corpus, but the measured boundaries retain unstable
cells and are scoped to the recorded platform.  Representative slices further
show that online descriptor generation is
visible at small degree but approaches planned FFR steady-state time at the
largest measured degrees.  On the four certified high-tap moduli,
the reduction advantage survives a complete Rabin irreducibility test and
widens relative to NTL over the measured degree range, providing a concrete
end-to-end computational-algebra result.

\section*{Software and Data Availability}
The implementation, deterministic manifests, reproduction scripts, and
original benchmark data accompany this version as ancillary files in its
\href{https://arxiv.org/src/2609.08608}{arXiv source distribution}.

\section*{Use of Artificial Intelligence}
OpenAI Codex (GPT-5.6 Sol) was used during the research and manuscript-preparation
workflow to assist with literature-search planning and synthesis,
mathematical exploration and statement auditing, software and test
development, benchmark-analysis and plotting code, and manuscript drafting
and editing.  The authors reviewed the resulting sources, arguments, code, and
reported data and assume full responsibility for the manuscript.

\appendix

\section{Coefficient-Algebra Extension}
\label{app:coefficient-algebra}

Let $R$ be a commutative characteristic-two algebra and let
\[
  g=x^m+\sum_{t=0}^{m-1}a_tx^t\in R[x]
\]
be monic.  Define $N$ as before on $R^m$, and put
\[
  U=\sum_{t=0}^{m-1}a_tN^{m-t},
  \qquad
  V(Y)=\sum_{t=0}^{m-1}a_t(x^tY\bmod x^m).
\]

\begin{theorem}[Coefficient-algebra factorization]
For every $k\ge0$,
\begin{equation}
  U^{2^k}=\sum_{t=0}^{m-1}a_t^{2^k}N^{2^k(m-t)}.
  \label{eq:coefficient-frobenius}
\end{equation}
If $U^{2^r}=0$, then
\[
  (I+U)^{-1}=\prod_{k=0}^{r-1}(I+U^{2^k}).
\]
Consequently, applying these factors to $H$ and returning $L+V(Y)$ computes
$A\bmod g$ for every $A=L+x^mH$ of degree below $2m$.
\end{theorem}
\begin{proof}
The summands $a_tN^{m-t}$ commute.  Frobenius is therefore additive on their
sum, and iterating it gives~\eqref{eq:coefficient-frobenius}.  The telescoping
identity
\[
  (I+U)\prod_{k=0}^{r-1}(I+U^{2^k})=I+U^{2^r}=I
\]
proves the inverse formula.  Finally,
\[
  A-gY=L+V(Y)+x^m\bigl(H+Y+U(Y)\bigr)
\]
in characteristic two, so $H=(I+U)Y$ leaves the unique degree-below-$m$
remainder $L+V(Y)$.
\end{proof}

This extension preserves the formal stage schedule; its actual nonzero support
may shrink.  In a nonreduced algebra a coefficient may satisfy
$a_t\ne0$ and $a_t^{2^k}=0$; for example,
$\varepsilon^2=0$ in $\Ftwo[\varepsilon]/(\varepsilon^2)$.  Thus the exact
binary expression $\Wfb$ is scheduled work in this setting, while an
implementation that removes zero coefficients may do less arithmetic.  In
odd prime characteristic $p$, the corresponding radix-$p$ factorization uses
$\sum_{j=0}^{p-1}(-U^{p^k})^j$, and the final remainder is $L-V(Y)$ rather
than the characteristic-two expression $L+V(Y)$.

\section{Validation and Additional Results}
\label{app:validation}

The deterministic algebraic suites use independent polynomial long division
or direct finite geometric series as their oracles.  Over $\Ftwo$, 20,000
random cases with $m\le128$ and all 299,592 monic-modulus/input pairs for
$m\le6$ pass.  Over $\Ftwo[\varepsilon]/(\varepsilon^2)$, 20,000 random cases
and all 266,304 cases for $m\le3$ pass; 13,573 random instances exhibit a
nonzero coefficient killed by a Frobenius power, motivating the support
qualification above.  A further 20,000 random $\mathbb F_4$ cases and 10,000
random $\mathbb F_3$ cases agree with independent inverse and reduction paths.
Exhaustive enumeration of all 262,125 nonempty constant-free supports for
$2\le m\le18$ finds no counterexample to the scheduled-work formula or its
coarse bound.

The repository entry point \texttt{make check} runs these suites together with
the native differential, boundary, masking, dense, and L\'opez--Dahab checks;
\texttt{make check-sanitize} separately runs sanitizer builds of the native
reducers and generated-reducer checks.  The exact artifact entry points are
\begin{center}
\small\ttfamily
make artifact-cost-model; make artifact-paper; make artifact-rabin;\\
make artifact-microbenchmark ARTIFACT\_CPU=0; make artifact-online.
\end{center}
The 67,576-row primary dataset and its 26 derived outputs are hash-locked.  From
a detached clean checkout at commit \texttt{8b21c90}, all 26 outputs, the
465-row Rabin artifact, and the 1,395-row, 45-support online artifact reproduced
with exact hash agreement while the full \texttt{make check} and
\texttt{make check-sanitize} suites passed.  Every
retained primary-dataset and Rabin row records its seed, taps, command,
compiler and gf2x paths, binary digest, affinity, and frequency policy.  The
online artifact records its applicable execution metadata, including command,
compiler flags, binary digest, affinity, and frequency policy; its
multiplication backend is not applicable.  This version contains no
secondary-platform result.

\end{document}